\documentclass{amsart}  
\usepackage{latexsym,amssymb,palatino,eulervm,color,graphicx,amsmath,amsthm}
\usepackage{amsfonts}
\usepackage{amsthm}
\usepackage{amsmath}
\usepackage{amsfonts}
\usepackage{latexsym}
\usepackage{amssymb}
\usepackage{amscd}
\usepackage[latin1]{inputenc}
\usepackage{verbatim}

\newtheorem{definition}{Definition}[section]
\newtheorem{theorem}[definition]{Theorem}
\newtheorem{lemma}[definition]{Lemma}
\newtheorem{corollary}[definition]{Corollary}
\newtheorem{proposition}[definition]{Proposition}
\theoremstyle{definition}

\newcommand\style{\mathcal }          

\newcommand{\M}{\style{M}}

\newcommand\A{{\style A}}
\renewcommand{\H}{\style{H}}

\newcommand\cstar{{\rm C}^*}                              
 
\newcommand\tr{ \operatorname{Tr} }

\def\mytimeA#1
 {%
   {%
      \count20=#1 %
      \count22=\count20
      \divide \count20 by 60
      \count21=\count20
      \multiply\count21 by -60
      \advance\count22 by \count21
      \ifnum\count20<12
           \def\tapm{ a.m.}%
      \else
           \def\tapm{ p.m.}%
      \fi
      \ifnum\count20=0
           \count20=12
      \fi
      \the\count20:%
      \ifnum\count22<10 0\fi
      \the\count22
   }%
 }

\usepackage{epsfig}
\usepackage{amsbsy,latexsym}
\usepackage{amsmath}
\usepackage{amssymb, mathrsfs}
\usepackage[mathscr]{eucal}
\usepackage{hyperref}
\usepackage{amsfonts}
\usepackage{amsthm}
\usepackage{amsmath}
\usepackage{amsfonts}
\usepackage{latexsym}
\usepackage{amssymb}
\usepackage{amscd}
\usepackage[latin1]{inputenc}
\usepackage{verbatim}
\usepackage{float}
\usepackage{xcolor}
\usepackage{ragged2e}
\usepackage[english]{babel}
\usepackage{enumitem}
\usepackage{multirow}
\usepackage{makecell} 
\usepackage{array}
\newcolumntype{?}{!{\vrule width 1pt}}

\newcommand{\im}{\rm{i}}

\everymath{\displaystyle}

\begin{document}
\title[Universality of Weyl Unitaries]{Universality of Weyl Unitaries}
\author[D.~Farenick]{Douglas Farenick \textsuperscript{1}}
\author[O.~Ojo]{Oluwatobi Ruth Ojo\textsuperscript{1}}
\author[S.~Plosker]{Sarah Plosker \textsuperscript{2,1}}

\thanks{\textsuperscript{1}Department of Mathematics and Statistics, University of Regina, Regina SK S4S 0A2, Canada}
\thanks{\textsuperscript{2}Department of Mathematics and Computer Science, Brandon University,
Brandon, MB R7A 6A9, Canada}

\date{ \today}

\begin{abstract} Weyl's unitary matrices, which were introduced in Weyl's 1927 
paper \cite{weyl1927} on group theory and quantum mechanics, 
are $p\times p$ unitary matrices given by the diagonal matrix whose entries 
are the $p$-th roots of unity and the cyclic shift matrix.
Weyl's unitaries, which we denote by $\mathfrak u$ and $\mathfrak v$, 
satisfy $\mathfrak u^p=\mathfrak v^p=1_p$ (the $p\times p$ identity matrix) and 
the commutation relation $\mathfrak u\mathfrak v=\zeta \mathfrak v\mathfrak u$, where $\zeta$ is a 
primitive $p$-th root of unity.  We prove that Weyl's unitary matrices are universal 
in the following sense: if $u$ and $v$ are
any $d\times d$ unitary matrices such that $u^p= v^p=1_d$ and 
$ u  v=\zeta vu$, then there exists a unital
completely positive linear map 
$\phi:\M_p(\mathbb C)\rightarrow\M_d(\mathbb C)$ such that $\phi(\mathfrak u)= u$ and 
$\phi(\mathfrak v)=v$. We also show, moreover, that any two pairs of $p$-th order 
unitary matrices that satisfy the Weyl commutation relation are completely order equivalent.

When $p=2$, the Weyl matrices are two of the three Pauli matrices from quantum mechanics. 
It was recently shown in \cite{farenick--huntinghawk--masanika--plosker2021} that 
$g$-tuples of Pauli-Weyl-Brauer unitaries are universal for all $g$-tuples of anticommuting 
selfadjoint unitary matrices; however, we show here that the analogous 
result fails for positive integers $p>2$.

Finally, we show that the Weyl matrices are extremal in their matrix range, 
using recent ideas from noncommutative convexity theory.
\end{abstract}


\maketitle
\section{Introduction}

With respect to a positive integer $p\geq2$ and a primitive $p$-th root of unity $\zeta$, 
a pair of $d\times d$ unitary matrices $u$ and $v$ satisfy the
\emph{Weyl relations} if 
\begin{equation}\label{comm rel}
u^p=v^p=1_d \mbox{ (the }d\times d\mbox{ identity matrix) and } uv=\zeta vu. 
\end{equation}
The relation $uv=\zeta vu$ is referred to as a \emph{Weyl commutation relation}.
The most immediate example of unitary matrices satisfying the Weyl relations comes from 
H.~Weyl's 1927 paper on quantum mechanics and 
group theory \cite[p.~32]{weyl1927}, in which $d=p$ and $u$ and $v$ are the 
$p\times p$ unitary matrices denoted herein
by $\mathfrak u$ and $\mathfrak v$, respectively, and are defined by
\begin{equation}\label{e:weyl_intro}
\mathfrak u=
\left[ \begin{array}{ccccc} 1 &&&& \\ & \zeta   &&& \\ && \zeta^2 && \\ &&&\ddots & \\ &&&& \zeta^{p-1} \end{array} \right] \mbox{ and }
\mathfrak v=
\left[ \begin{array}{ccccc} 0 &&&& 1 \\ 1 & 0 &&& \\ &1& 0 && \\ &&\ddots&\ddots & \\ &&&1& 0 \end{array} \right].
\end{equation}
In the case where $p=2$, the Weyl unitaries
$\mathfrak u$ and $\mathfrak v$ are two (namely, $\sigma_Z$ and $\sigma_X$) of the three Pauli matrices: 
\begin{eqnarray*}
\sigma_X=\begin{bmatrix}0&1\\1&0\end{bmatrix},\quad 
\sigma_Y=\begin{bmatrix}0&-\im\\\im&0\end{bmatrix}, \quad
\sigma_Z=\begin{bmatrix}1&0\\0&-1\end{bmatrix}.
\end{eqnarray*}
For this reason, the Weyl unitaries are often viewed as a generalised form of the Pauli matrices. 
(In particular, Weyl's paper \cite{weyl1927} introduces his $p\times p$ unitaries after a fulsome discussion
of the Pauli matrices.)

Weyl's unitary matrices $\mathfrak u$ and $\mathfrak v$ 
in (\ref{e:weyl_intro}) are not just examples of unitary matrices satisfying the relations (\ref{comm rel});
they are, in fact, universal matrices for all such pairs of unitary matrices. 
That is, as we shall prove herein, if $u$ and $v$ are $d\times d$ unitary matrices 
such that $u^p=v^p=1_d$,  
and $u v=\zeta vu$, 
then there exists a unital completely positive linear map 
$\Phi:\M_p(\mathbb C)\rightarrow\M_d(\mathbb C)$ such that $\Phi(\mathfrak u)= u$
and $\Phi(\mathfrak v)= v$, where $\M_n(\mathbb C)$ denotes the algebra of complex $n\times n$ matrices.

Returning to the case $p=2$, the Weyl matrices $\mathfrak u$ and $\mathfrak v$ are the Pauli matrices
$\sigma_Z$ and $\sigma_X$, respectively. The third Pauli matrix, $\sigma_Y$, is obtained from the other two 
via the equation $\sigma_Y=i\sigma_X\sigma_Z$. Any unitary
$u\in\M_d(\mathbb C)$ for which $u^2=1_d$ must be selfadjoint; and the condition $uv=\zeta vu$ is equivalent, in the case $p=2$, to $uv=-vu$. Therefore, the 
relations (\ref{comm rel}) for pairs of unitaries extend easily to $g$-tuples of unitaries, thereby describing
$g$-tuples of anticommuting selfadjoint unitary matrices.

However, for $p>2$,
extending the relations in
(\ref{comm rel}) from two unitaries to a higher number of unitaries, say $g$, 
has a number of additional considerations. 
The case of $g=2$ illustrates the situation already: suppose that $u$ and $v$ unitaries that
satisfy the relations in (\ref{comm rel}). The commutation relation $uv=\zeta vu$ implies that $vu=\zeta^{-1} uv$;
it is only in the case where $p=2$ that $\zeta^{-1} =\zeta$, and so one is required, for $p>2$, 
to account for the change in the scalar if one swaps the order of the unitary product.

If $\zeta$ is a primitive $p$-th root of unity, then $\zeta^k$ is also a primitive $p$-th root of unity
for all $k\in\mathbb N$ that are not divisible by $p$---or, stated more conveniently, for any nonzero $k\in\mathbb Z_p$.
Thus, one might have commutation relations for unitaries that involve $\zeta$ and some its powers.

Such considerations lead to the following definition.

\begin{definition}\label{comm rel many}
Suppose that $\zeta\in\mathbb C$ is a primitive $p$-th root of unity.
The \emph{Weyl commutation relations} for a $g$-tuple $u=(u_1,\dots,u_g)$ of $d\times d$
unitary matrices that satisfy $u_k^p=1_d$, for each $k$,
are given by the equations
\begin{equation}\label{e:comm rel many}
 u_ku_{\ell}=\zeta^{c_{k\ell}}u_\ell u_k, \mbox{ for all } k,\ell\in\{1,\dots,g\},
\end{equation}
for some skew-symmetric matrix $C=[c_{k\ell}]_{k,\ell=1}^g\in\M_g(\mathbb Z_p)$.
If $c_{k\ell}=1$ whenever $k<\ell$, then the commutation relations in (\ref{e:comm rel many}) are said to be \emph{simple}.
\end{definition}

Motivated by the fact that the Pauli matrix $\sigma_Y$ is a scalar 
multiple of the product $\sigma_X\sigma_Z$, one may mimic the construction by considering $w=\lambda uv$,
where $\lambda\in\mathbb C$ and $u$ and $v$ are $d\times d$ unitary 
matrices such that $u^p=v^p=1_d$ and $uv=\zeta vu$.
The condition $w^p=1$
imposes a requirement upon $\lambda$. Indeed, as $vu=(1/\zeta)uv$,
\[
\begin{array}{rcl}
(uv)^p&=&u(vu)^{p-1}v=(1/\zeta)^{p-1}u^2(vu)^{p-2}v^2 \\
         &=& (1/\zeta)^{p-1}(1/\zeta)^{p-2}u^3(vu)^{p-3}v^3 \\
         &=& \cdots \\
         &=& (1/\zeta)^{1+2+\cdots+(p-1)}u^pv^p = (1/\zeta)^{\frac{(p-1)p}{2}}\,1.
\end{array}
\]
Thus, the condition $w^p=1$ implies that $\lambda^p=\zeta^{\frac{(p-1)p}{2}}$, and 
so $\lambda$ is uniquely determined: $\lambda=\zeta^{\frac{p-1}{2}}$.
In this case, we have the following commutation relations (in addition to $uv=\zeta vu$):
\[
uw=\zeta wu \mbox{ and } vw= \zeta^{-1} wv.
\]
Thus, if $(u_1,u_2,u_3)$ were given by $(u,v,w)$, then the matrix $C$ for the 
Weyl commutation relations would be
\[
C=\left[ \begin{array}{ccc} 0&1&1 \\ -1&0&-1 \\ -1&1&0 \end{array}\right].
\]
On the other hand, the Weyl commutation relations for $(u_1,u_2,u_3)=(u,w,v)$ are simple, as in this case the matrix $C$ is
\[
C=\left[ \begin{array}{ccc} 0&1&1 \\ -1&0&1 \\ -1&-1&0 \end{array}\right].
\]

\begin{definition} If $p\geq 3$, then 
the \emph{simple Weyl unitary matrices} are the three $p\times p$ unitary matrices
$\omega_a$, $\omega_b$, and $\omega_c$ defined by
\[
\omega_a=\mathfrak u,\quad \omega_b=\zeta^{\frac{p-1}{2}}\mathfrak u\mathfrak v,\quad \omega_c=\mathfrak v.
\]
The triple
$\mathfrak W=(\omega_a,\omega_b,\omega_c)$ is called the \emph{simple Weyl triple}.
\end{definition}

In this paper, we establish a result that is stronger than the assertion 
concerning the universality of the Weyl matrices $\mathfrak u$ and $\mathfrak v$. Specifically, we
show that any two unitary matrices with the relations (\ref{comm rel}) are universal -- not just the
Weyl pair. This result is best phrased in terms of complete order equivalence, which is essentially a special
form of isomorphism in the operator system category (see \cite{Paulsen-book} for an overview of the theory of
operator systems and completely positive linear maps).

To explain complete order equivalence, suppose that $x=(x_1,\dots,x_g)$ is a $g$-tuple of 
complex $d\times d$ matrices $x_k$. The \emph{operator system of $x$} is the
linear space $\mathcal O_x$ of $d\times d$ matrices defined by
\[
\mathcal O_x=\mbox{\rm Span}_{\mathbb C}\{1_d,x_k,x_k^*\,|\,k=1,\dots,g\}.
\]
If $x=(x_1,\dots,x_g)$ and $y=(y_1,\dots,y_g)$ are $g$-tuples of 
matrices $x_j\in\M_{d_1}(\mathbb C)$ and $y_\ell\in M_{d_2}(\mathbb C)$,
then $x$ and $y$ are said to be \emph{completely order equivalent}, 
denoted by $x\simeq_{\rm ord} y$, if there exists a linear
isomorphism $\phi:\mathcal O_x\rightarrow\mathcal O_y$ such that
\begin{enumerate}
\item[{a)}] $\phi$ and $\phi^{-1}$ are unital completely positive (ucp) linear maps, and
\item[{b)}] $\phi(u_k)=y_k$, for each $k=1,\dots, g$.
\end{enumerate}

If $a=(a_1,\dots,a_g)$ is a $g$-tuple of $d\times d$ matrices, and 
if $\gamma: \mathbb C^n\rightarrow \mathbb C^d$ is a linear transformation,
then $\gamma ^* a \gamma$ denotes the $g$-tuple of $n\times n$ matrices given by
\[
\gamma ^* a \gamma= (\gamma^*a_1\gamma,\dots,\gamma^*a_g\gamma).
\]
If $n=d$ and if $b=\gamma^* a \gamma$ for a unitary $\gamma$, then this situation is denoted by 
\[
a\simeq_{\rm u} b
\]
and we say that  $b$ is \emph{unitarily equivalent} to $a$.

While it is clear that
\[
x\simeq_{\rm u} y \Longrightarrow x\simeq_{\rm ord} y,
\]
the converse does not hold in general, for in the complete order equivalence problem for  
$g$-tuples $x$ and $y$ of matrices,
the matrix dimensions $d_1$ and $d_2$ need not be equal. Hence, complete order equivalence
is weaker than unitary equivalence. 

An answer to the question of when two operator $g$-tuples $x$ and $y$ are 
completely order equivalent has been given 
by Davidson, Dor-On, Shalit, and Solel \cite{davisdon--dor-on--shalit--solel2017}: 
$x\simeq_{\rm ord} y$ if and only if $W(x)=W(y)$, where,
for a $g$-tuple $z=(z_1,\dots,z_g)$ of operators $z_k$, $W(z)$ 
denotes the \emph{matrix range of $z$}, which is the sequence
\[
W(z)=\left( W^n(z)\right)_{n\in\mathbb N}
\]
of subsets $W^n(z)$ in $\M_n(\mathbb C)\times \cdots \times \M_n(\mathbb C)$ ($g$ copies) defined by
\begin{equation}\label{e:n-mr}
W^n(z)=\left\{ \left(\psi(z_1),\dots,\psi(z_g)\right)\,|\,\psi:\mathcal O_z\rightarrow\M_n(\mathbb C)\mbox{ is a ucp map}\right\}.
\end{equation}
This result of Davidson, Dor-On, Shalit, and Solel is especially 
useful in cases where the matrix ranges can be described; however, such
descriptions are not always available.

We shall find it convenient to use the following definition. 

\begin{definition} A unitary matrix $u$ is of \emph{$p$-th order} if $u^p$ is the identity matrix.
\end{definition}

Lastly, some notation:
if $\mathcal S$ is a nonempty subset of $\M_d(\mathbb C)$, then $\mbox{Alg}(\mathfrak S)$ denotes the complex associative subalgebra of 
$\M_d(\mathbb C)$ generated by $\mathcal S$. In particular, it is a classical result that $\M_p(\mathbb C)=\mbox{Alg}(\mathcal S)$, 
if $\mathcal S$ is the set consisting of the two Weyl unitary matrices 
$\mathfrak u$ and $\mathfrak v$ \cite{schwinger1960}.

\section{Universality of the Weyl Unitary Matrices $\mathfrak u$ and $\mathfrak v$}

The eigenvalues of the Weyl unitary matrices $\mathfrak u$ and $\mathfrak v$ are 
precisely the $p$-roots of unity, which is the maximal spectrum of any $p$-th order unitary matrix.
It so happens that this spectral property is a consequence of the Weyl commutation relation.

\begin{lemma}\label{trace and spectrum} 
If $u$ and $v$ are $p$-th order $d\times d$ unitary matrices in which $uv=\zeta vu$, then
\begin{enumerate}
\item $p$ divides $d$,
\item $\tr(u)=\tr(v)=\tr(uv)=0$, and
\item every $p$-root of unity is an eigenvalue of $u$ and an eigenvalue of $v$.
\end{enumerate}
\end{lemma}

\begin{proof}  Applying the determinant to the relation $uv=\zeta vu$
leads to
\[
\det(u)\det(v)=\det(uv)=\zeta^d\det(vu)=\zeta^{d}\det(u)\det(v),
\]
which implies that $\zeta^{d}=1$. Because $\zeta$ is primitive, the positive integer $p$ divides $d$.

The relation $uv=\zeta vu$  is equivalent to $v^*uv=\zeta u$. Thus, by applying the trace,
we have that $\tr(u)=\zeta\tr(u)$, which is true only if $\tr(u)=0$. 
Analogous
reasoning leads to $\tr(v)=0$. Finally, 
\[
\tr(vu)=\tr(uv)=\zeta \tr(vu),
\]
which leads to $\zeta =1$ or $\tr(uv)=0$. As the former cannot hold, it must be that the latter does.

We next show that every $p$-th root of unity is an eigenvalue of $u$. Because the spectrum of a matrix is invariant
under unitary similarity, the relation $v^*uv=\zeta u$ implies that $\zeta\lambda$ is an eigenvalue of $u$ whenever
$\lambda$ is an eigenvalue of $u$. Applying this observation to the eigenvalue $\zeta\lambda$, we see that
$\zeta^2\lambda$ is an eigenvalue of $u$. Hence, by iteration of the argument, $\zeta^k\lambda$ is an eigenvalue
of $u$ for every $k=1,\dots,p$. Because the map $\alpha\mapsto\alpha\lambda$ is a bijection of $\mathbb C$ onto itself,
the spectrum of $u$ must contain at least $p$ elements. But on the other hand, the
spectrum of $u$ cannot contain more than $p$
elements; hence, the spectrum of $u$ must coincide with the set of $p$-th roots of unity. 
A similar argument applies
to $v$.
\end{proof}

The next lemma determines the set of all $p$-th order unitaries $v$, 
given a unitary $u$ with the spectral structure described in Lemma \ref{trace and spectrum}, for which $uv=\zeta vu$.

\begin{lemma}\label{l1} If $d=pn$ and if $u,v\in\M_d(\mathbb C)$ are $p$-th order
unitary matrices such that $uv=\zeta vu$, then there exist a unitary matrix  $y\in\M_d(\mathbb C)$
and unitary matrices $v_2,\dots,v_p\in\M_n(C)$ such that $y^*uy=\tilde u$ and $y^*vy=\tilde v$, where 
\begin{equation}\label{matrix form}
\tilde u=\left[ \begin{array}{ccccc} 1_n &&&& \\ & \zeta1_n &&& \\ &&\zeta^21_n && \\ &&&\ddots & \\ &&&& \zeta^{p-1}1_n\end{array}\right] \mbox{ and }
\tilde v=\left[ \begin{array}{ccccc}0&0&\cdots&& v_1 \\ v_2&0&\ddots&&\vdots \\  0& v_3 & 0 && \\ \vdots & & \ddots& \ddots & \vdots \\ 0 &\cdots &0 & v_p &0 
\end{array}\right] ,
\end{equation}
where $1_n$ denotes the $n\times n$ identity matrix, and where $v_1$ is given by  $v_1=v_2^*v_3^*\cdots v_p^*$.
\end{lemma}

\begin{proof}  Lemma \ref{trace and spectrum} indicates that every 
$p$-th root of unity is a spectral element of $u$ and $\tr(u)=0$. Therefore, using $d=pn$, each eigenvalue $\zeta^k$ of
$u$ must appear with multiplicity $n$. Hence, $u$ admits a diagonalisation such as that given in (\ref{matrix form}).

Now let $\tilde v=y^*vy$; in accordance with the structure of $\tilde u$, the operator $\tilde v$ can be expressed as a $p\times p$
matrix of $n\times n$ matrices $v_{ij}$. The relation $\tilde u\tilde v=\zeta^c\tilde v\tilde u$ holds only if the operators $v_{ij}=0$ whenever
$(i,j)\not\in\{(1,p), (k,k-1)\,|\,k=2,\dots,p\}$. Thus, denote $v_{1p}$ by $v_1$ and $v_{k,k-1}$ by $v_k$. Because 
$\tilde v^*\tilde v=\tilde v\tilde v^*=1_d$, each $v_k$ satisfies $v_k^*v_k=v_kv_k^*=1_n$.
Furthermore, $\tilde v ^p=1_d$ implies that $v_pv_{p-1}\dots v_2v_1=1_n$, and so 
$v_1=v_2^*v_3^*\cdots v_p^*$.
\end{proof}

As a consequence of the result above, it is possible to construct a path-connected set 
of $p$-th order unitaries $v$ such that $uv=\zeta vu$ from a certain 
$p$-th order unitary 
matrix $u$.

\begin{corollary} If $u\in\M_d(\mathbb C)$ is a $p$-th order unitary matrix such that
every $p$-th root of unity is an eigenvalue of $u$ of multiplicity $d/p$, then the set of all $p$-th order unitary matrices
$v\in\M_d(\mathbb C)$
for which $uv=\zeta vu$ 
is homeomorphic to the Cartesian product of $p-1$ copies of the unitary group $\mathcal U_{\frac{d}{p}}$.
\end{corollary}

The next result is crucial for establishing 
an explicit complete order equivalence in our main result, Theorem \ref{pairs} below.

\begin{lemma}\label{l2} If $d=pn$ and if $u,v\in\M_d(\mathbb C)$ 
are unitaries of the form (\ref{matrix form}), for some unitaries $v_2,\dots,v_p\in\M_n(\mathbb C)$,
then a $p\times p$ matrix 
$Z=[z_{ij}]_{i,j=1}^p$ of $n\times n$ matrices $z_{ij} $ is an element of $\mbox{\rm Alg}(\{u,v\})$
if and only if there exist $p^2$ scalars $\lambda_{ij}$ such that:
\begin{equation}\label{double commutant}
\begin{array}{rcl} 
z_{ii}&=&\lambda_{ii}1_n, \mbox{ for all }i ;\\
z_{1p}&=& \lambda_{1p}v_1 \mbox{ and } z_{k,k-1}=\lambda_{k,k-1}v_k, \mbox{ for all }k=2,\dots,p; \\
z_{p1}&=&\lambda_{p1}v_1^*\mbox{ and } z_{k,k+1}=\lambda_{k,k+1}v_{k+1}^*, \mbox{ for all }k=1,\dots,p-1; \\
z_{ij}&=&\lambda_{ij} (v_i\cdots v_{j+1}), \mbox{ for all }i>j \mbox{ with } |i-j|\geq2; \\
z_{ij}&=&\lambda_{ij} (v_{i+1}^*\cdots v_j^*), \mbox{ for all }i<j \mbox{ with } |i-j|\geq2.
\end{array}
\end{equation}
\end{lemma}

\begin{proof}
First note that if a matrix $b$ commutes with a unitary $w$, then $b$ also commutes with $w^{-1}=w^*$, implying that $b^*$ commutes with $w$; hence, the set $\{w\}'$ of matrices
commuting with $w$ is closed under the adjoint $*$. Second, as the inverse $w^{-1}$ of a unitary matrix $w$ is a polynomial in $w$, then
algebra generated by one or more unitary matrices is $*$-closed. Hence, by von Neumann's Double Commutant Theorem,
\[
\mbox{\rm Alg}(\{u,v\})=\{u,v\}'',
\]
where $\mathcal S''$ denotes the double commutant $(\mathcal S')'$ of a set $\mathcal S$.
We begin, therefore, by showing that 
\begin{equation}\label{commutant}
\{u,v\}'=\left\{ \left( \bigoplus_{k=1}^{p-1} (v_p\cdots v_{k+1})^*x(v_p\cdots v_{k+1})\right)\bigoplus x\,|\,x\in\M_n(\mathbb C)\right\}.
\end{equation}

To this end,
suppose that $X=[x_{ij}]_{i,j=1}^p\in \M_d(\mathbb C)$, where $x_{ij}\in\M_n(\mathbb C)$ for all $i$ and $j$, commutes with $u$ and $v$.
Then $Xu=uX$ implies that $\zeta^{i-1}x_{ij}=\zeta^{j-1}x_{ij}$ for all $i,j$, and so $x_{ij}=0$ for all $i$ and $j$ with $j\not=i$. The equation $Xv=vX$
yields $x_{11}v_1=v_1x_{pp}$ and $v_kx_{k-1,k-1}=x_{kk}v_k$ for $k=2,\dots,p$. Thus,
\begin{equation}\label{farm}
x_{pp}=v_1^*x_{11}v_1, \; x_{11}=v_2^*x_{22}v_2,\;x_{22}=v_3^*x_{33}v_3, \;\cdots,\; x_{p-1,p-1}=v_p^*x_{pp}v_p,
\end{equation}
which yields
\begin{equation}\label{farm boy}
x_{kk}=(v_{k}\cdots v_1)x_{pp}(v_1^*\cdots v_{k}^*),\mbox{ for all }k=1,\dots,p.
\end{equation}
Therefore, once $x_{pp}$ is specified, all other diagonal entries of $X$ are determined.

Suppose next that $Z=[z_{ij}]_{i,j=1}^p\in \M_d(\mathbb C)$, where $z_{ij}\in\M_n(\mathbb C)$ for all $i$ and $j$, commutes with every $X\in\{u,v\}'$.
Write the diagonal entries of $X\in\{u,v\}'$ as $x_1,\dots,x_p$, and recall from equation (\ref{farm boy})
that
\[
x_{k}=(v_{k}\cdots v_1)x_{p}(v_1^*\cdots v_{k}^*),\mbox{ for all }k=1,\dots,p.
\]
The equation $XZ=ZX$ implies that $x_iz_{ij}=z_{ij}x_j$ for all $i,j$. For $i=j$, these relations imply that 
$z_{ii}$ commutes with every $n\times n$ matrix; hence, $z_{ii}=\lambda_{ii}1_n$, for some 
$\lambda_{ii}\in\mathbb C$.   

If $k\in\{2,\dots,p\}$, then $x_kz_{k,k-1}=z_{k,k-1}x_{k-1}$ is written, using the equations (\ref{farm}), as
\[
x_kz_{k,k-1}=z_{k,k-1}(v_k^*x_kv_k),
\]
and so $x_k(z_{k,k-1}v_k^*)=(z_{k,k-1}v_k^*)x_k$. As this holds for all $x_k\in\M_n(\mathbb C)$, $z_{k,k-1}v_k^*=\lambda_{k,k-1}1_n$,
for some $\lambda_{k,k-1}\in\mathbb C$; that is, $z_{k,k-1}=v_k$, for $k=2,\dots,p$. The same type of 
calculation leads to $z_{1p}=\lambda_{1p}v_1$.

Similarly, if $k\in\{1,\dots,p-1\}$, then $x_kz_{k,k+1}=z_{k,k+1}x_{k+1}$ is, using the equations (\ref{farm}), equivalent to the equation
\[
x_{k+1}(v_{k+1}z_{k,k+1})=(v_{k+1}z_{k,k+1})x_{k+1},
\]
which implies that $z_{k,k+1}=\lambda_{k,k+1}v_{k+1}^*$, as the equation above must hold for all $x_{k+1}$.
Likewise, $z_{p1}=v_1^*$.

Consider now the entries of $Z$ for which $|i-j|\geq2$. Assume first the cases where
$i>j$. By equations (\ref{farm}),
\[
\begin{array}{rcl}
x_i&=&(v_i\cdots v_{j+1})(v_j\cdots v_1)x_p(v_1^*\cdots v_j^*)(v_{j+1}^*\cdots v_i^*) \\ && \\
&=& (v_i\cdots v_{j+1}) x_j (v_{j+1}^*\cdots v_i^*).
\end{array}
\]
Thus, the equation $x_iz_{ij}=z_{ij}x_j$ is
\[
x_iz_{ij}=z_{ij}(v_i\cdots v_{j+1})^* x_i (v_i\cdots v_{j+1}),
\]
which is equivalent to 
\[
x_i\left( z_{ij}(v_i\cdots v_{j+1}) ^*\right) = \left(z_{ij}(v_i\cdots v_{j+1}) ^* \right)x_i.
\]
As $x_i\in\M_n(\mathbb C)$ can be arbitrary, $z_{ij}=\lambda_{ij}1_n$ for some $\lambda_{ij}\in\mathbb C$, which implies that
\[
z_{ij}=\lambda_{ij} (v_i\cdots v_{j+1}), \mbox{ for all }i>j \mbox{ with } |i-j|\geq2.
\]

In the cases where $|i-j|\ge2$ and $i<j$, the same type of arguments lead to 
\[
z_{ij}=\lambda_{ij} (v_{i+1}^*\cdots v_j^*), \mbox{ for all }i<j \mbox{ with } |i-j|\geq2,
\]
which completes the proof.
\end{proof}

As an example of the matrix structure indicated 
Lemma \ref{l2}, consider the case where $p=5$. Select any $n\in\mathbb N$ and unitary matrices 
$v_1,\dots,v_4\in M_n(\mathbb C)$,
and set $v_5=(v_4v_3v_2v_1)^*$. Thus, 
\[
u=\left[ \begin{array}{ccccc} 1_n &&&& \\ & \zeta 1_n &&& \\ && \zeta^21_n && \\ &&&\zeta^31_n & \\ &&&& \zeta^41_n \end{array} \right]
\mbox{ and }
v=\left[ \begin{array}{ccccc} 0 &&&& v_1 \\ v_2 & 0 &&& \\ &v_3& 0 && \\ &&v_4&0 & \\ &&&v_5& 0 \end{array} \right].
\]
By Lemma \ref{l2}, $z\in\mbox{\rm Alg}(\{u,v\})$ if and only if
\[
z=
\left[ \begin{array}{ccccc} 
\lambda_{11}\,1_n &\lambda_{12}\,v_2^* &\lambda_{13}(v_2^*v_3^*)&\lambda_{14}(v_2^*v_3^*v_4^*)& \lambda_{15}\,v_1 \\ 
\lambda_{21} \,v_2  & \lambda_{22}1_n &\lambda_{23}\,v_3^*&\lambda_{24}(v_3^*v_4^*)& \lambda_{25}(v_3^*v_4^*v_5^*)\\ 
\lambda_{31}(v_3v_2)&\lambda_{32}\,v_3& \lambda_{33}1_n &\lambda_{34}\,v_4^*&\lambda_{35} (v_4^*v_5^*)\\ 
\lambda_{41}(v_4v_3v_2)&\lambda_{42}(v_4v_3)&\lambda_{43}\,v_4&\lambda_{44}1_n&\lambda_{45}\,v_5^* \\ 
\lambda_{51}(v_5v_4v_3v_2)&\lambda_{52}(v_5v_4v_3)&\lambda_{53}(v_5v_4)&\lambda_{54}\,v_5& \lambda_{55}1_n 
\end{array} \right],
\]
for some $\lambda_{ij}\in\mathbb C$.

Let us also make note of the following useful consequence of Lemma \ref{l2}.

\begin{corollary}\label{weyl irr}
If $u,v\in\M_p(\mathbb C)$ 
are unitaries of the form (\ref{matrix form}), for some complex numbers $v_2,\dots,v_p$,
then $\{u,v\}'=\{\lambda\,1_p\,|\,\lambda\in\mathbb C\}$ and $\mbox{\rm Alg}(\{u,v\})=\M_p(\mathbb C)$.
\end{corollary}

The structure of $\mbox{\rm Alg}(\{u,v\})$ provided by Lemma \ref{l2} gives an explicit $*$-isomorphism of matrix algebras:

\begin{proposition}\label{m=2 iso}
If $d=pn$ and if $u,v\in\M_d(\mathbb C)$ 
are unitaries of the form (\ref{matrix form}), for some unitaries $v_2,\dots,v_p\in\M_n(\mathbb C)$,
then the function $\rho:\M_p(\mathbb C)\rightarrow\M_d(\mathbb C)$ defined by
\[
\rho\left([\lambda_{ij}]_{i,j=1}^p\right) = [z_{ij}]_{i,j}^p,
\]
where the matrices $z_{ij} \in\M_n(\mathbb C)$ are given as in equations (\ref{double commutant}), is a unital 
$*$-isomorphism of $\M_p(\mathbb C)$ and $\mbox{\rm Alg}(\{u,v\})$.
\end{proposition} 

\begin{proof} By Lemma \ref{l2}, the function $\rho$ maps $\M_p(\mathbb C)$ onto $\cstar(u,v)$. Furthermore, $\rho$ is plainly linear and $\ker\rho=\{0\}$, 
and so $\rho$ is a linear isomorphism. It remains to show that $\rho$ is a $*$-homomorphism. Equations (\ref{double commutant}) indicate
that $\rho(\Lambda^*)=\rho(\Lambda)^*$, for all $\Lambda\in \M_p(\mathbb C)$, and so the multiplicativity of $\rho$ is the only property left to confirm.

To this end, let $A=[\alpha_{ij}]_{i,j}$ and $B=[\beta_{ij}]_{i,j}$ be elements of $\M_p(\mathbb C)$. We shall compare the entries of $\rho(A)\rho(B)$ with those of $\rho(AB)$.
If, for example, $i>j$ and $|i-j|\geq 2$, then the entries of row $i$ of $\rho(A)$ are
\[
\alpha_{i1}(v_i\cdots v_2), \quad \alpha_{i2}(v_i\cdots v_3), \quad \cdots,\quad \alpha_{ii}1_n, \quad \alpha_{i,i+1}v_{i+1}^*, \quad\cdots, \alpha_{ip}(v_{i+1}^*\cdots v_p^*),
\]
while the entries of column $j$ of $\rho(B)$ are
\[
\beta_{1j}(v_2^*\cdots v_j^*), \quad \beta_{2j}(v_3^*\cdots v_j^*),\quad\cdots, \beta_{jj}1_n, \quad \beta_{j+1,j} v_p, \quad\cdots, \beta_{pj}(v_p\cdots v_{j+1}).
\]
Therefore, the $(i,j)$-entry of $\rho(A)\rho(B)$ is $\displaystyle\sum_{k=1}^p \alpha_{ik}\beta_{kj}(v_i\cdots v_{j+1})$, which is the $(i,j)$-entry of $\rho(AB)$.
The arguments for all other choices of $i$ and $j$ are similar.
\end{proof}

\begin{theorem}\label{pairs} If two $p$-th order unitary matrices 
$u$ and $v$ satisfy the Weyl commutation relation $uv=\zeta vu$, then
$( u,  v)\simeq_{\rm ord}(\mathfrak u,\mathfrak v)$.
\end{theorem}

\begin{proof} If $u,v\in\M_d(\mathbb C)$, then Lemma \ref{trace and spectrum} asserts that $p$ divides $d$ and that there is a unitary matrix $y\in\M_d(\mathbb C)$ 
such that $y^*uy$ and $y^*vy$ are the matrices in (\ref{matrix form}). As the map $x\mapsto y^*xy$ 
is a $*$-automorphism of $\M_d(\mathbb C)$, it is enough to assume that $u$ and $v$
are in this form already. In that regard, the isomorphism $\rho:\M_p(\mathbb C)\rightarrow\M_d(\mathbb C)$ in Theorem \ref{m=2 iso} satisfies $\rho(\mathfrak u)=u$ and
$\rho(\mathfrak v)=v$. Furthermore, as $\rho$ is a unital $*$-isomorphism, its restriction $\phi$ to the operator system generated by the Weyl matrices $\mathfrak u$
and $\mathfrak v$ has the property that both $\phi$ and $\phi^{-1}$ are completely positive. 
\end{proof}

\begin{corollary}\label{main result Weyl} The Weyl unitary matrices $\mathfrak u$ and $\mathfrak v$ are universal for the commutation relation $uv=\zeta vu$.
\end{corollary}

\begin{proof} Suppose that $u$ and $v$ are $p$-th order $d\times d$ unitary matrices 
that satisfy the Weyl commutation relation $uv=\zeta vu$. By Theorem \ref{pairs}, the linear isomorphism $\phi:\mathcal O_{(\mathfrak u, \mathfrak v)}\rightarrow\mathcal O_{(u,v)}$
is completely positive. Viewing $\phi$ as a ucp from the operator 
system $\mathcal O_{(\mathfrak u, \mathfrak v)}$ into the matrix algebra $\M_d(\mathbb C)$, the map $\phi$
admits a ucp extension $\Phi$ to $\M_p(\mathbb C)$, by the Arveson Extension Theorem 
\cite{arveson1969,Paulsen-book}. Clearly the map $\Phi$ sends $\mathfrak u$ to $u$
and $\mathfrak v$ to $v$.
\end{proof}

\begin{corollary}\label{pairs2} If two $p$-th order unitary matrices $u$ and $v$ satisfy the Weyl commutation relation $uv=\zeta vu$, then
the pair $(u,v)$ is universal for all $p$-th order unitary matrices that satisfy the Weyl commutation relation.
\end{corollary}

Our final observation is that if the Weyl commutation relations are satisfied by $p$-th order
$p\times p$ unitaries, then these unitaries must be unitarily equivalent to the Weyl unitaries.

\begin{corollary}\label{pairs} If two $p$-th order $p\times p$
unitary matrices $u$ and $v$ satisfy the Weyl commutation relation $uv=\zeta vu$, then
$( u,  v)\simeq_{\rm u}(\mathfrak u,\mathfrak v)$.
\end{corollary}

\begin{proof} By hypothesis, there is a unital completely positive bijection 
$\phi:\mathcal O_{(\mathfrak u,\mathfrak v)}\rightarrow \mathcal O_{( u,  v)}$ in which $\phi(\mathfrak u)=u$,
$\phi(\mathfrak v)=v$, and $\phi^{-1}$ is completely positive. Let $\Phi$ and $\Phi_1$ be extensions 
of $\phi$ and $\phi^{-1}$, respectively, to ucp maps $\M_p(\mathbb C)\rightarrow\M_p(\mathbb C)$.
The ucp map $\Phi_1\circ\Phi$ fixes every element of the irreducible operator system 
$\mathcal O_{(\mathfrak u,\mathfrak v)}$, and so by Arveson's Boundary Theorem
\cite{arveson1972,farenick2011b}, $\Phi_1\circ\Phi$ is the identity map of $\M_p(\mathbb C)$.
Hence, $\Phi$ is a ucp map of $\M_p(\mathbb C)$ with a completely positive inverse, which by Wigner's Theorem
implies that $\Phi$ is a unitary similarity transformation $x\mapsto w^*xw$ for some unitary $w\in\M_p(\mathbb C)$.
\end{proof}

\section{Weyl-Brauer Unitaries}

Consider the Weyl triple $\mathfrak W=(\omega_a, \omega_a , \omega_c)$ of  $p\times p$ unitary matrices, where
\[
\omega_a=\mathfrak u, \quad \omega_b=\zeta^{\frac{p-1}{2}}\mathfrak u \mathfrak v , \quad \omega_c= \mathfrak v,
\]
and where $\mathfrak u$ and $\mathfrak v$ are the Weyl unitary matrices. Recall that
the triple $\mathfrak W=(\omega_a, \omega_a , \omega_c)$ satisfies the simple Weyl commutation relations.

Observe that $\mathfrak v=\zeta^{\frac{1-p}{2}}\mathfrak u^{p-1}\mathfrak w$, and so 
$\mathfrak v$ is in the associative algebra generated by $\omega_a$ and $\omega_b$. 
Because $\M_p(\mathbb C)$ is generated as an algebra by $\mathfrak u$ and $\mathfrak v$,
this means that the algebra generated $\omega_a$ and $\omega_b$ is also $\M_p(\mathbb C)$. 
Hence, 
\[
\mbox{\rm Alg}\left( \mathcal Q_{1,-}  \right) = \mbox{\rm Alg}\left( \mathcal Q_{1,-}  \right) =\M_p(\mathbb C),
\]
where
$\mathcal Q_1=\{\omega_a,\omega_b,\omega_c\}$ and $\mathcal Q_{1,-}=\{\omega_a,\omega_b\}$ .
       
As in \cite[Definition 6.63]{Watrous-book} and 
by adapting the method of the proof of Theorem 4.3 
in \cite{farenick--huntinghawk--masanika--plosker2021}, 
we shall invoke an iteration whereby we produce, from $m$ invertible matrices $x_1,\dots,x_m$, a set of $m+2$ invertible matrices:
\[
x_1\otimes 1_p,\dots,x_{m-1}\otimes 1_p, x_m\otimes \omega_a, x_m\otimes \omega_b, x_m\otimes \omega_c.
\]
Specifically, in taking $x_1$, $x_2$, and $x_3$ to be $\omega_a$, $\omega_b$, and $\omega_c$, respectively, the iteration yields a set
$\mathcal Q_2\subset\M_p(\mathbb C)\otimes\M_p(\mathbb C)$ of $5$ elements:
\[
\begin{array}{rcl}
\mathcal Q_2&=&\{ \omega_a\otimes 1,  \omega_b\otimes 1, \omega_c\otimes \omega_a, \omega_c\otimes \omega_b, \omega_c\otimes \omega_c\} \\
&=& \mathcal Q_{2,-} \cup \{\omega_c\otimes \omega_c\},
\end{array}
\]
where $\mathcal Q_{2,-} =\mathcal Q_2 \setminus \{\omega_c\otimes \omega_c\}$. The matrices in $\mathcal Q_2$ satisfy the Weyl commutation relations $\tilde u\tilde v=\zeta \tilde v\tilde u$ when
$\tilde u$ is selected before $\tilde v$ in $\mathcal Q_2$ and the set $\mathcal Q_2$ is considered as an ordered list.

Another iteration of the construction generates a set $\mathcal Q_3$ consisting of $7$ elements:
\[
\mathcal Q_3= \mathcal Q_{3,-} \cup \{\omega_c\otimes \omega_c\otimes\omega_c\},
\]
where
\[
\mathcal Q_{3,-}=\{  
\omega_a\otimes 1 \otimes 1 ,
 \omega_b\otimes 1 \otimes 1 ,
 \omega_c\otimes \omega_a \otimes 1,
 \omega_c\otimes \omega_b \otimes 1, 
 \omega_c\otimes \omega_c \otimes \omega_a ,
\omega_c\otimes \omega_c \otimes \omega_b\}.
\]
Once again, the matrices in $\mathcal Q_3$ satisfy the Weyl commutation relations $\tilde u\tilde v=\zeta \tilde v\tilde u$ when
$\tilde u$ is selected before $\tilde v$ in $\mathcal Q_3$ and the set $\mathcal Q_3$ is considered as an ordered list.

Repeated iteration produces, for each positive integer $k$, a set $\mathcal Q_{k,-}$ of cardinality $2k$ and a 
set $\mathcal Q_k$ with one additional element, namely
\[
\mathcal Q_k= \mathcal Q_{k,-} \cup \left\{\bigotimes_1^k\omega_c \right\},
\]
such that the matrices in $\mathcal Q_k$ satisfy the Weyl commutation relations  $\tilde u\tilde v=\zeta \tilde v\tilde u$ when
$\tilde u$ is selected before $\tilde v$ in $\mathcal Q_k$ and the set $\mathcal Q_k$ is considered as an ordered list.

The $2k$ elements of $\mathcal Q_{k,-}$ consist of $k$ pairs such that, in the order given by the iterative construction,
the product of each pair is a product tensor in which all factors are the identity matrix and one tensor factor is $\omega_a\omega_b$. More specifically,
if
\[
\mathcal Q_{k,-}=\{z_1,z_2,z_3,z_4,\dots,z_{2k-1},z_{2k}\} \subset \bigotimes_1^k \M_p(\mathbb C),
\]
then
\[
\begin{array}{rcl}
z_1z_2&=&(\omega_a\omega_b)\otimes 1 \otimes 1 \cdots \otimes 1 = \zeta^{\frac{1-p}{2}}\left( \omega_c \otimes 1 \otimes 1 \cdots \otimes 1\right) \\
z_3z_4&=&1 \otimes (\omega_a\omega_b)\otimes 1\cdots\otimes 1 = \zeta^{\frac{1-p}{2}}\left( 1 \otimes \omega_c\otimes 1\cdots \otimes 1\right) \\
\vdots &=& \vdots \\
z_{2k-1}z_{2k}&=& 1 \otimes 1 \otimes 1 \cdots \otimes (\omega_a\omega_b)= \zeta^{\frac{1-p}{2}} \left(1 \otimes 1 \otimes 1 \otimes 1 \cdots \otimes \omega_c\right). 
\end{array}
\]
Hence,
\[
\bigotimes_1^k\omega_c = (\zeta^{\frac{1-p}{2}})^{-k} \displaystyle\prod_{j=1}^k w_{2j-1}w_{2j} \in \mbox{\rm Alg}\left(\mathcal Q_{k,-} \right) ,
\]
which shows that
\[
\mbox{\rm Alg} ({\mathcal Q_{k,-}}) =\mbox{\rm Alg} ({\mathcal Q_k}),
\]
for every $k\in \mathbb N$. 

We now show that $\mbox{\rm Alg} ({\mathcal Q_{k,-}}) =\mbox{\rm Alg} ({\mathcal Q_k})=\displaystyle\bigotimes_1^k\M_p(\mathbb C)$. Of course, it is sufficient to show this
for  $\mbox{\rm Alg} ({\mathcal Q_{k,-}})$.
The claim holds for $k=1$ because  $\mathcal Q_{1,-}=\{\omega_a,\omega_b\}$ generates $\M_p(\mathbb C)$. Looking at the case $k=2$, 
\[
\mathcal Q_{2,-}=\{ \omega_a\otimes 1,  \omega_b\otimes 1, \omega_c\otimes \omega_a, \omega_c\otimes \omega_b \} .
\]
The algebra generated by $\{\omega_a\otimes 1,  \omega_b\otimes 1\}$ consists of all matrices of the form $s\otimes 1$, for $s\in\M_p(\mathbb C)$, whereas the
algebra generated by $\{\omega_c\otimes \omega_a, \omega_c\otimes \omega_a\}$ consists of all matrices of the form $\omega_c\otimes t$, for $t\in\M_p(\mathbb C)$. 
Because the set of all products of matrices of these two types is the set of all elementary tensors in $\M_p(\mathbb C)\otimes \M_p(\mathbb C)$, we see that the claims holds for $k=2$. In general, using induction,
if we consider matrices of the form given by the construction
\[
x_1\otimes 1_p,\dots,x_{m-1}\otimes 1_p, x_m\otimes \omega_a, x_m\otimes \omega_b, x_m\otimes \omega_c,
\]
where the algebra generated by invertible matrices $x_1 ,\dots,x_{m-1}$ is a full matrix algebra $\M_d(\mathbb C)$, then our argument here shows that 
$x_1\otimes 1_p,\dots,x_{m-1}\otimes 1_p$ generate matrices of the form $s\otimes 1_p$ while $x_m\otimes \omega_a$ and $ x_m\otimes \omega_b$ generate
all matrices of the form $x_m\otimes t$. Thus, collectively, these matrices generate $\M_d(\mathbb C)\otimes \M_p(\mathbb C)$.

The construction above proves the following theorem.

\begin{theorem}\label{tensor constr} For every positive integer $k$ there exist $p$-th order unitaries $u_1,\dots,u_{2k+1}$ that satisfy the
simple Weyl commutation relations and are such that
\[
\mbox{\rm Alg}(\{u_1,\dots,u_{2k}\})=\mbox{\rm Alg}(\{u_1,\dots,u_{2k}, u_{2k+1}\})=\displaystyle\bigotimes_1^k\M_p(\mathbb C).
\]
\end{theorem}
 
As mentioned in \cite{farenick--huntinghawk--masanika--plosker2021}, it is often the case that
spin systems arise in quantum theory; for this reason, the following modification of the construction above is worth
a brief mention.

\begin{theorem} If $\H$ is an infinite-dimensional complex Hilbert space, then there exists a countable sequence
$\{u_n\}_{n\in\mathbb N}$ of $p$-th order unitary operators $u_n$ such that $u_ku_\ell=\zeta u_\ell u_k$ whenever
$k<\ell$ and such that the norm-closed algebra generated by the sequence $\{u_n\}_{n\in\mathbb N}$ is isomorphic
to the C$^*$-algebra $\displaystyle\bigotimes_1^\infty \M_p(\mathbb C)$.
\end{theorem}

\begin{proof} 
Let $\H=\displaystyle\bigotimes_1^\infty  \mathbb C^p$, which is the direct limit of the finite-dimensional Hilbert
spaces $\H_k=\displaystyle\bigotimes_1^k \mathbb C^p$. On each $\H_k$ construction the Weyl-Brauer unitaries,
and then form the tensor product of these unitaries with infinitely many copies of the $p\times p$ identity matrix
so as to produce unitary operators $u_1,\dots,u_{2k}$ on $\H$ of order $p$ that satisfy the simple Weyl commutation 
relations. This construction also shows that the algebra $\A_k$ generated by $u_1,\dots,u_{2k}$ is isomorphic to
$\displaystyle\bigotimes_1^k\M_p(\mathbb C)$ and that $\A_k$ is a unital subalgebra of $\A_{k+1}$. Hence, the 
norm-closed algebra generated by $\{u_n\}_{n\in\mathbb N}$ coincides with the norm-closure of
$\displaystyle\bigcup_{k\in\mathbb N}\A_k$, which is precisely 
$\displaystyle\bigotimes_1^\infty \M_p(\mathbb C)$.
\end{proof}
 
\section{Weyl-Brauer Unitaries Are Not Universal}

Although the Pauli-Weyl-Brauer matrices are universal 
for selfadjoint anticommuting unitaries \cite{farenick--huntinghawk--masanika--plosker2021}, 
the analogous result fails for $p>2$.

\begin{proposition}\label{three is bad} Assume that $p\geq 3$ and let 
$\mathfrak W=(\omega_a,\omega_b, \omega_c)$ denote the triple of simple Weyl-Brauer matrices. 
Let
$x,y,z\in\M_p(\mathbb C)$ be the $p$-th order unitary matrices given by 
$x=\omega_a$, $z=\omega_c$, and 
\begin{equation}\label{e:w}
y=
\left[ \begin{array}{ccccccc} 0 &&&&&& 1\\ \zeta^2 & 0 &&&&& \\ 
&\zeta^3& 0 && &&\\ 
&&\ddots& \ddots &&& \\ &&&\zeta^{p-2}&0 && \\
&&&&\zeta^{p-1}&\ddots & \\ 
&&&&&\zeta^{(1-p)/2} & 0 \end{array} \right].
\end{equation}
The triple $(x,y,z)$ satisfies the simple Weyl commutation relations, but 
there does not exist any unital completely positive linear map 
$\phi:\M_p(\mathbb C)\rightarrow\M_p(\mathbb C)$
in which $\phi(\omega_a)=x$, $\phi(\omega_b)=y$, and $\phi(\omega_c)=z$.
\end{proposition}

\begin{proof} If $\lambda_p$ denotes the $(p,p-1)$-entry of $y$, then the condition $y^p=1_p$ implies
that
\[
\lambda_p=\left(1\cdot\zeta\cdot\zeta^2\cdots\zeta^{p-1}\right)^{-1}
= \zeta^{ -\sum_{k=1}^{p-1}k} =\zeta^{(1-p)/2};
\]
furthermore,
matrix multiplication confirms that the triple $(x,y,z)$ satisfies the simple 
Weyl commutation relations. Hence, $(x,y,z)$
is a simple Weyl triple.

Assume that a
unital completely positive linear map 
$\phi:\M_p(\mathbb C)\rightarrow\M_p(\mathbb C)$
in which $\phi(\omega_a)=x$, $\phi(\omega_b)=y$, and $\phi(\omega_c)=z$ does exist.
The fixed point set $\mathfrak F_\phi=\{s\in\M_d(\mathbb C)\,|\,\phi(s)=s\}$
of $\phi$ is an operator system that contains the Weyl unitaries
$\omega_a=\mathfrak u$ and $\omega_c=\mathfrak v$; because $\{\mathfrak u, \mathfrak v\}'=\mathbb C\,1_d$, the
operator system $\mathfrak F_\phi$ is irreducible.

Let $\psi:\mathfrak F_\phi\rightarrow\M_p(\mathbb C)$ be the ucp map $\psi(s)=s$, for all $s\in \mathfrak F_\phi$.
Thus, $\phi$ is a ucp extension of $\psi$ from $\mathfrak F_\phi$ to $\M_p(\mathbb C)$. However, as 
$\mathfrak F_\phi$ is irreducible, $\psi$ has unique completely positive extension to $\M_p(\mathbb C)$, 
by Arveson's Boundary Theorem \cite{arveson1972,farenick2011b}.
Therefore, $\phi$ can only be the identity map, as the identity map on $\M_d(\mathbb C)$ is one ucp extension of $\psi$. 
However, as $\phi(\omega_a)=x\not=\omega_a$,
$\phi$ is not the identity map.
Hence, this contradiction leads us to conclude that
$\phi$ is not completely positive.
\end{proof}

 \begin{corollary}\label{neg result} The Weyl-Brauer unitaries are not universal for $g$-tuples (where $g\geq 3$)
of $p$-th order unitaries that satisfy the simple Weyl commutation relations.
\end{corollary}

\begin{proof} If universality were to hold for some $g>3$, then it would need to hold for $g=3$. However, Proposition
\ref{three is bad} indicates
that universality fails for $g=3$.
\end{proof}

\section{The Matrix Range of the Weyl Unitaries}

The work of Arveson \cite{arveson1972} and 
Davidson, Dor-On, Shalit, and Solel \cite{davisdon--dor-on--shalit--solel2017}
demonstrate the role of the matrix range in questions such as those we have considered herein.
For this reason, it is of interest to consider the matrix range of the Weyl unitaries, especially in connection
with the geometry of the matrix range from the perspective of 
free convexity \cite{evert--helton--klep--mccullough2018}.

\begin{definition}
A sequence $K=\left(K_n\right)_{n\in\mathbb N}$ of subsets $K_n$
in the Cartesian product $\M_n(\mathbb C)^g$ of $g$ copies of $\M_n(\mathbb C)$ is 
\emph{matrix convex} if   
\begin{equation}\label{e:mc}
\sum_{j=1}^m \gamma_j^*a_j\gamma_j \in K_n  
\end{equation}
for all $m\in\mathbb N$, all $a_j\in K_{n_j}$, and all linear transformations $\gamma_j:\mathbb C^n\rightarrow\mathbb C^{n_j}$ for which
\begin{equation}\label{e:mc2}
\sum_{j=1}^m \gamma_j^* \gamma_j =1_n.
\end{equation}
Linear transformations $\gamma_j$ that satisfy (\ref{e:mc2}) are called \emph{matrix convex coefficients} and elements of the form 
(\ref{e:mc}) are called \emph{matrix convex combinations} of the elements $a_j$.
\end{definition}

We are interested in the following notions \cite{evert--helton--klep--mccullough2018} of extremal element in the context of matrix convexity.

\begin{definition} If $K=\left(K_n\right)_{n\in\mathbb N}$ is matrix convex, where $K_n\subseteq\M_n(\mathbb C)^g$ 
for each $n$, then an element $b\in K_n$ is:
\begin{enumerate}
\item an \emph{absolute extreme point} of $K$ if whenever $b$ is a matrix convex combination
(\ref{e:mc}) of elements $a_j\in K_{n_j}$ such that each matrix convex coefficient $\gamma_j$ is nonzero, then, for each $j$, either (i)
$n_j=n$ and $a_j\simeq_{\rm u}b$ or (ii) $n_j>n$ and there exists a $c_j$ such that $a_j\simeq_{\rm u}b\oplus c_j$;
\item a \emph{matrix extreme point} of $K$ if whenever $b$ is a matrix convex combination
(\ref{e:mc}) of elements $a_j\in K_{n_j}$ such that each matrix convex coefficient $\gamma_j$ is surjective, then, for each $j$, $n_j=n$ and $a_j\simeq_{\rm u}b$.
\end{enumerate}
\end{definition}

\begin{theorem}\label{matrix extreme point} 
The Weyl pair $(\mathfrak u,\mathfrak v)$ is a matrix extreme point of its matrix range.
\end{theorem}

\begin{proof} 
Let $b=(\mathfrak u,\mathfrak v)$ and suppose that $b=\displaystyle\sum_{j=1}^m \gamma_j^*a_j\gamma_j$
for some surjective matrix convex coefficients $\gamma_j$ and matrix pairs $a_j=(a_{j1},a_{j2})\in W^{n_j}(b)$. For each $j$ there is a 
ucp map $\psi_j:\mathcal O_{(\mathfrak u, \mathfrak v)}\rightarrow \M_{n_j}(\mathbb C)$ such that $a_{j1}=\psi_j(\mathfrak u)$ and
$a_{j2}=\psi_j(\mathfrak v)$. Let $\Psi_j$ be a ucp extension of $\psi_j$ to a ucp map $\Psi_j:\M_p(\mathbb C)\rightarrow\M_{n_j}(\mathbb C)$, for each $j$, and let
$\Phi:\M_p(\mathbb C)\rightarrow\M_p(\mathbb C)$ be given by
\[
\Phi=\displaystyle\sum_{j=1}^m \gamma_j^*\Psi_j\gamma_j.
\]
Note that $\Phi_{\vert\mathcal O_{(\mathfrak u, \mathfrak v)}}$ is the identity map on $\mathcal O_{(\mathfrak u, \mathfrak v)}$; hence, in considering the 
identity map in $\mathcal O_{(\mathfrak u, \mathfrak v)}$
as a ucp map from $\mathcal O_{(\mathfrak u, \mathfrak v)}$ into $\M_p(\mathbb C)$, $\Phi$ is a ucp extension of that map. Because
the operator system $\mathcal O_{(\mathfrak u, \mathfrak v)}$ is irreducible, Arveson's Boundary Theorem \cite{arveson1972,farenick2011b}
implies that $\Phi$ is the identity map on $\M_p(\mathbb C)$.
That is,
\[
\mbox{\rm id}_{\M_p(\mathbb C)}=\displaystyle\sum_{j=1}^m \gamma_j^*\Psi_j\gamma_j.
\]
Now because the identity map of $\M_p(\mathbb C)$ is a pure matrix state of $\M_p(\mathbb C)$, it is also a matrix extreme point of its matrix state space \cite{farenick2000}. Hence, 
$n_j=p$ for every $j$ and there are unitaries $w_j\in\M_p(\mathbb C)$ such that $\Psi_j(x)=w_j^*xw_j$ for every $x\in\M_p(\mathbb C)$. In particular, $\mathfrak u = w_j^*a_{j1}w_j$ and 
$\mathfrak v = w_j^*a_{j2}w_j$ for all $j$, which yields $a_j\simeq_{\rm u}(\mathfrak u, \mathfrak v)=b$ for all $j$.
\end{proof}
 
\begin{theorem}\label{absolute extreme point} The Weyl pair $(\mathfrak u,\mathfrak v)$ is an absolute extreme point of its matrix range.
\end{theorem}

\begin{proof} Let $b=(\mathfrak u,\mathfrak v)$ and suppose that, for some $\ell\in\mathbb N$, the pair 
\[
(a_1,a_2)=
\left( \left[\begin{array}{cc} \mathfrak u & r_1 \\ s_1 & t_1\end{array}\right], \left[\begin{array}{cc} \mathfrak v & r_2 \\ s_2 & t_2\end{array}\right]\right)
\in W^{p+\ell}(b),
\]
for some matrices $r_i,s_i, t_i$ of appropriate sizes. We claim that the matrix pair above $(a_1,a_2)$
can be an element of $W^{p+\ell}(b)$ only if the off-diagonal
matrices $r_i$ and $s_i$ are zero, for $i=1,2$. The reason for this is straightforward. Because each $a_i$ is a ucp image of a Weyl unitary, 
$\|a_i\|\leq 1$ for $i=1,2$. Therefore, no row or column in $a_i$ can have norm (in $\mathbb C^{p+\ell}$) exceeding 1. However, because each
row of $\mathfrak u$ and $\mathfrak v$ has exactly one nonzero entry and this entry is of modulus 1, a nonzero entry in $r_1$ or $r_2$ would 
cause a row in $a_1$ or $a_2$ to have norm exceeding 1, which we noted cannot happen. Thus, $r_1$ and $r_2$ are zero matrices.
Using a similar argument for the columns, 
we deduce that $s_1$ and $s_2$ are also zero matrices.

In the language of matrix convexity, the previous paragraph proves that the 
pair $(\mathfrak u, \mathfrak v)$ is an Arveson extreme 
point of its matrix range $W(b)$. By \cite[Theorem 1.1(3)]{evert--helton--klep--mccullough2018}, if an Arveson extreme point of a matrix convex set
is irreducible, then it is an absolute extreme point. Since the commutant $\{\mathfrak u,\mathfrak v\}'$ is $1$-dimensional, the Weyl pair is irreducible and, hence,
an absolute extreme point of its matrix range.
\end{proof}

The proofs of Theorems \ref{matrix extreme point} and \ref{absolute extreme point} 
extend beyond Weyl pairs to all Weyl-Brauer unitaries, once it has
been shown that the Weyl-Brauer matrices generate irreducible operator systems. 
The details are left to the reader (using, if one wishes, the method of proof in 
\cite{farenick--huntinghawk--masanika--plosker2021} that showed the irreducibility 
of the operator system generated by the Pauli-Weyl-Brauer unitaries). However, 
at the very least, we state below the version of this theorem for the three basic 
Weyl-Brauer unitaries $\omega_a$, $\omega_b$, and $\omega_c$.

\begin{theorem} If $\mathfrak W=(\omega_a,\omega_b,\omega_c)$ is the Weyl-Brauer triple, then $\mathfrak W$ is a matrix extreme point and an absolute
extreme point of its matrix range.
\end{theorem}

\section*{Acknowledgements}
This work was supported, in part, by the 
NSERC Discovery Grant program, the 
Canada Foundation for Innovation, and the Canada Research Chairs program.

\bibliographystyle{amsplain}
\bibliography{doug-refs}

\providecommand{\bysame}{\leavevmode\hbox to3em{\hrulefill}\thinspace}
\providecommand{\MR}{\relax\ifhmode\unskip\space\fi MR }
\providecommand{\MRhref}[2]{%
  \href{http://www.ams.org/mathscinet-getitem?mr=#1}{#2}
}
\providecommand{\href}[2]{#2}
\begin{thebibliography}{10}

\bibitem{arveson1969}
William Arveson, \emph{Subalgebras of {$C\sp{\ast} $}-algebras}, Acta Math.
  \textbf{123} (1969), 141--224. \MR{MR0253059 (40 \#6274)}

\bibitem{arveson1972}
\bysame, \emph{Subalgebras of {$C\sp{\ast} $}-algebras. {II}}, Acta Math.
  \textbf{128} (1972), no.~3-4, 271--308. \MR{MR0394232 (52 \#15035)}

\bibitem{davisdon--dor-on--shalit--solel2017}
Kenneth~R. Davidson, Adam Dor-On, Orr~Moshe Shalit, and Baruch Solel,
  \emph{Dilations, inclusions of matrix convex sets, and completely positive
  maps}, Int. Math. Res. Not. IMRN (2017), no.~13, 4069--4130. \MR{3671511}

\bibitem{evert--helton--klep--mccullough2018}
Eric Evert, J.~William Helton, Igor Klep, and Scott McCullough, \emph{Extreme
  points of matrix convex sets, free spectrahedra, and dilation theory}, J.
  Geom. Anal. \textbf{28} (2018), no.~2, 1373--1408. \MR{3790504}

\bibitem{farenick2000}
Douglas Farenick, \emph{Extremal matrix states on operator systems}, J. London
  Math. Soc. (2) \textbf{61} (2000), no.~3, 885--892. \MR{1766112
  (2001e:46103)}

\bibitem{farenick2011b}
\bysame, \emph{Arveson's criterion for unitary similarity}, Linear Algebra
  Appl. \textbf{435} (2011), no.~4, 769--777. \MR{2807232 (2012d:15002)}

\bibitem{farenick--huntinghawk--masanika--plosker2021}
Douglas Farenick, Farrah Huntinghawk, Adili Masanika, and Sarah Plosker,
  \emph{Complete order equivalence of spin unitaries}, Linear Algebra Appl.
  \textbf{610} (2021), 1--28.

\bibitem{Paulsen-book}
Vern Paulsen, \emph{Completely bounded maps and operator algebras}, Cambridge
  Studies in Advanced Mathematics, vol.~78, Cambridge University Press,
  Cambridge, 2002. \MR{MR1976867 (2004c:46118)}

\bibitem{schwinger1960}
Julian Schwinger, \emph{Unitary operator bases}, Proc. Nat. Acad. Sci. U.S.A.
  \textbf{46} (1960), 570--579. \MR{115648}

\bibitem{Watrous-book}
John Watrous, \emph{The theory of quantum information}, Cambridge University
  Press, Cambridge, 2018.

\bibitem{weyl1927}
H.~Weyl, \emph{Quantenmechanik und {G}ruppentheorie}, Z. Physik A \textbf{46}
  (1927), no.~1, 1--46.

\end{thebibliography}

\end{document}